\documentclass[11pt]{article}

\usepackage{amssymb,amsthm,amsmath}
\usepackage[mathscr]{eucal}

\newcommand{\Irr}{{\rm Irr}}
\newcommand{\wt}{{\rm wt}}
\newcommand{\base}{{\rm base}}
\newcommand{\rank}{{\rm rank}}
\newcommand{\cf}{{\rm Cf}}

\newcommand{\be}{{\bf e}}

\newcommand{\bu}{{\bf u}}

\newcommand{\bx}{\mbox{$\mathbf{x}$}}

\newcommand{\by}{{\bf y}}

\usepackage{bm} 

\newcommand{\C}{\mbox{$\mathbb C$}}

\newcommand{\Z}{\mbox{$\mathbb Z$}}

\newcommand{\inv}{\ensuremath{^{-1}}}

\newcommand\inner[1]{\langle \, #1  \, \rangle}
\newcommand{\cart}[3]{\ensuremath{{#1}_{#2} \times \cdots \times {#1}_{#3}}}

\newcommand{\pperp}{\ensuremath{{\rm \perp}\kern-.60em {\rm \perp} }}
\newcommand{\npperp}{\ensuremath{{\rm \perp}\kern-.60em {\rm \not\perp} }}
\newcommand\set[1]{\{1, \ldots, #1 \}}

\newcommand{\chihat}{\ensuremath{\hat{\chi}}}
\newcommand{\gbar}{\ensuremath{\bar{g}}}
\newcommand{\ybar}{\ensuremath{\bar{y}}}

\newtheorem{theorem}{Theorem}[section]
\newtheorem{lemma}[theorem]{Lemma}

\theoremstyle{definition}  
\newtheorem{example}[theorem]{Example}
\newtheorem{remark}[theorem]{Remark}
\newtheorem{definition}[theorem]{Definition}

\begin{document}
\author{Jay H. Beder
\\Department of Mathematical Sciences\\University of Wisconsin-Milwaukee\\
P.O. Box 413\\Milwaukee, WI 53201-0413\\
{\tt beder@uwm.edu}\\  \\Jesse S. Beder
\\Department of Mathematics\\University of Illinois at Urbana-Champaign\\
1409 W. Green Street \\Urbana, IL 61801\\
{\tt beder@math.uiuc.edu}\\  }

\title{Generalized wordlength patterns and strength}
\date{}

\maketitle

\vspace{-1cm}

\begin{abstract}
Xu and Wu (2001) defined the \emph{generalized wordlength pattern} $(A_1, \ldots, A_k)$ of an arbitrary fractional factorial design (or orthogonal array) on $k$ factors.  They gave a coding-theoretic
proof of the property that the design has strength $t$ if and only if $A_1 = \cdots = A_t = 0$.  The quantities $A_i$ are defined in terms of characters of cyclic groups, and so one might seek a  direct character-theoretic proof of this result.  We give such a proof, in which the specific group structure (such as cyclicity) plays essentially no role.  Nonabelian groups can be used if the counting function of the design satisfies one assumption, as illustrated by a couple of examples. 
\end{abstract}

{\footnotesize {\bf Key words.} Fractional factorial design; group
character; Hamming weight; multiset; orthogonal array; strength }

{\footnotesize {\bf AMS(MOS) subject classification.}
Primary: 62K15; 
Secondary:
05B15, 
20C15, 
62K05 
}

\section{Introduction}\label{intro-sec}
A \emph{fractional factorial design} is a multisubset $D$ of a
finite Cartesian product $G = \cart{G}{1}{k}$, that is, a set of
elements of $G$, the element \bx\ possibly repeated with some
multiplicity $O(\bx)$.
We will say that $D$ is \emph{based on} $G$, and refer to $O$ as the \emph{counting}
or \emph{multiplicity function}\footnote{It is called the \emph{indicator function} of $D$ by a number of authors -- for example, in \cite{ChengYe04}.} of $D$.  In statistical terminology, the
set $G_i$ indexes the levels of the $i$th factor in an experiment, and $G$ is the set of treatment combinations. The treatment combinations used in the design are referred to as \emph{runs}, and the
number of runs in the design, counting multiplicities, is
  \begin{equation} \label{|D|-eq} |D| = \sum_{\bx \in G} O(\bx).
  \end{equation}

Xu and Wu \cite{XuWu01} associated to a design $D$ a $k$-tuple
$(A_1(D), \ldots, A_k(D))$, called its \emph{generalized wordlength
pattern}, defined as follows.  If $G_i$ has $s_i$ elements, we take $G_i = \Z_{s_i}$, the additive cyclic group of integers modulo $s_i$.  This makes $G$ an abelian group.  To each $u \in \Z_s$ we associate a complex-valued function $\chi_u$ on $\Z_s$ such that
\begin{equation} \label{cyclic_char-eq}\chi_u(x) = \xi^{ux}, \end{equation}
where $\xi$ is a primitive $s$th root of unity (say $\xi = e^{2\pi
i/s}$).  For $\bu = (u_1, \ldots, u_k)$ and $\bx = (x_1, \ldots,
x_k) \in G$, we let
  \begin{equation} \label{char_prod-eq} \chi_{\bu}(\bx) = \prod_i \chi_{u_i}(x_i),
  \end{equation}
and define the \emph{$J$-characteristics}\footnote{When $s_1 = \cdots = s_k = 2$, the quantities $\chi_{g}(D)$ reduce to the \emph{$J$-characteristics} of Deng and Tang \cite{DengTang99}.  We are following Ai and Zhang \cite{AiZhang04} in using the same term for these quantities in the general case.} of the design to be the quantities
  \begin{equation} \label{J-char-eq} 
  \chi_{\bu}(D) =
  \sum_{\bx \in G} O(\bx) \overline{\chi_{\bu}(\bx)},
  \end{equation}
the bar denoting the complex conjugate.  This formula departs superficially from that given in~\cite{XuWu01}.  The introduction of the conjugate does not change the value of $\chi_{\bu}(D)$ since the choice of $\xi$ in (\ref{cyclic_char-eq}) is arbitrary and may be replaced by $\xi = e^{-2\pi i/s}$.  The factor $O(\bx)$ makes it explicit that each summand is repeated according to its multiplicity.

Finally, the ``generalized wordlengths" are given by
  \begin{equation} \label{genwordlength}
  A_j(D) = N^{-2} \sum_{\wt(\bu) = j} |\chi_{\bu}(D)|^2 \qquad \mbox{for} \; j =
  1, \ldots, k, \end{equation}
where $N = |D|$ is defined as in (\ref{|D|-eq}) and $\wt(\bu)$ is the
\emph{Hamming weight} of \bu, that is, the number of non-zero components of \bu. For the statistical meaning of the generalized wordlength pattern, the reader is referred to
\cite{XuWu01}.

The design $D$ may also be viewed as an orthogonal array,
particularly if its runs are displayed in matrix form, say as
columns of a $k \times N$ matrix.  Xu and Wu \cite[Theorem 4(ii)]{XuWu01} use a coding-theoretic
result to show that
$A_1(D) = \cdots = A_t(D) = 0$
iff $D$ has strength $t$.  They note
in passing that the functions (\ref{cyclic_char-eq}) and
(\ref{char_prod-eq}) are group characters, which might lead us to expect a
character-theoretic proof of this result. Providing such a proof is
the purpose of this paper.

Using a suggestive idea from \cite{Bierbrauer95},
we first reexpress the numbers $A_j(D)$ in terms of certain Fourier
coefficients.

The functions $\chi_u$ in (\ref{cyclic_char-eq}) are the \emph{irreducible characters} of the
group $\Z_s$, and so the functions $\chi_{\bu}$ are the irreducible characters of $G$.  They form
an orthonormal basis of the set of all functions from $G$ to \C\ under the inner product
  \begin{equation} \label{inner-eq} \inner{\phi,\psi}
  = \frac{1}{|G|}\sum_{\bx \in G} \phi(\bx)\overline{\psi(\bx)}.
  \end{equation}
If we express $O$ in this basis as
  \[ O = \sum_{\bu \in G} \mu_{\bu} \chi_{\bu}, \]
then its Fourier coefficients $\mu_{\bu}$ satisfy
    \[ \mu_{\bu} = \inner{O,\chi_{\bu}}
  = \frac{1}{|G|} \sum_{\bx \in G} O(\bx) \overline{\chi_{\bu}(\bx)}
  = \frac{1}{|G|} \chi_{\bu}(D), \]
so that the generalized wordlengths (\ref{genwordlength}) are given
by
  \[ A_j(D) = N^{-2} \sum_{\wt(\bu) = j} |\chi_{\bu}(D)|^2
  =  \frac{|G|^2}{N^2}\sum_{\wt(u)=j} |\mu_{\bu}|^2. \]
To establish our claim, we need to show that $D$ has strength $t$
iff $\mu_{\bu} = 0$ for all \bu\ such that $1 \le \wt(\bu) \le t$.

It turns out that this result does not depend on the fact that the
groups $G_i$ are cyclic, or even abelian, although in the nonabelian
case we will need to recast the concept of weight and to impose one restriction on $O$.  We recast the main result in Section~\ref{restatement-sec}, and give the proof in Section~\ref{proof-sec}.  Background on character theory and on strength is given in Sections~\ref{character-sec} and \ref{strength-sec}.  We conclude with two examples illustrating the restriction on $O$ in the nonabelian case.

There are many excellent expositions of character theory, and we will sometimes mention known results without citation. We will often refer to \cite{Isaacs76}; other texts include
\cite{Ledermann87} and~\cite{Serre77}. \vspace{1cm}

\textbf{Notation and terminology}. As already indicated, the complex conjugate of $z$ will be denoted by $\bar{z}$.  We denote the complex numbers by \C, the integers modulo $s$ by $\Z_s$, the cardinality of a set $E$ by $|E|$, and vectors ($k$-tuples) by boldface.  The set of complex-valued functions on $G$ will be written $\C^G$.

All groups are finite. The identity element of a group will generally be denoted by $e$.

When $G = \cart{G}{1}{k}$ is a direct product of groups, the \emph{Hamming weight} $\wt(\bu)$ of an element $\bu \in G$ will be defined as the number of nonidentity components of \bu.  Here we have modified the usual definition of Hamming weight as $G_i$ may have no zero symbol. Each $G_i$ may be identified with a subgroup of $G$, namely the subgroup $e_1 \times \cdots e_{i-1} \times G_i \times e_{i+1} \times \cdots \times e_k$ where $e_j$ is the identity of $G_j$.  A similar identification holds for $\cart{G}{{i_1}}{{i_m}}$ where $1 \le i_1 < \cdots < i_m \le k$.  For such subgroups it will be useful to introduce the following terminology.

\begin{definition} \label{factorial-def} If $H = \cart{G}{{i_1}}{{i_m}}$, we call $H$ a \emph{factorial subgroup} of $G$.  The number $m$ will be called the \emph{rank} of $H$.  The \emph{factorial complement} of $H$ in $G$ is $\prod_{i \notin I} G_i$, where $I = \{i_1, \ldots, i_m \}$. \end{definition} 

\section{Characters} \label{character-sec}
We will deal only with complex-valued characters.  We refer the reader to a treatment of character theory for more detail, and simply quote the results that we will need.

The set of characters on the group $G$ is closed under pointwise addition, and contains a finite set $\Irr(G)$ that generates it in the sense that every character on $G$ is a unique linear combination of characters in $\Irr(G)$ with nonnegative integer coefficients.  The characters in $\Irr(G)$ are called \emph{irreducible}.  Among them is the \emph{principal character} $\chi \equiv 1$.  The irreducible characters of the cyclic group $\Z_s$ are given by (\ref{cyclic_char-eq}), while for an abelian group $G$ they are the homomorphisms from $G$ to the multiplicative group $\C^*$ \cite[Corollary 2.6]{Isaacs76}.

%

If $G = \cart{G}{1}{k}$ and $\chi_i$ is a character on $G_i$, then
  \[ \chi(\bx) = \prod_i \chi_i(x_i) \qquad (\bx = (x_1, \ldots,x_m) \in G)   \]
defines a character on $G$, and $\chi \in \Irr(G)$ iff $\chi_i \in \Irr(G_i)$ for all $i$ \cite[Theorem~4.21]{Isaacs76}.

\begin{definition}  For $\chi$ a character of $G$, $\ker(\chi) = \{g \in G: \chi(g) = \chi(e) \}$, where $e$ is the identity of $G$.
\end{definition}
One can show\footnote{A representation of $G$ is a homomorphism, and $\ker(\chi)$ is the kernel of the representation affording $\chi$. When $G$ is abelian and $\chi$ is irreducible, $\ker(\chi)$ is the kernel of the homomorphism $\chi$.} that $\ker(\chi)$ is a normal subgroup of $G$.  The number $\chi(e)$ is a positive integer, called the \emph{degree} of $\chi$.

A character on $G$ is a \emph{class function}, that is, a function that is constant on the conjugacy
classes of $G$.  Let
  \[\cf(G)= \mbox{the set of class functions from $G$ to \C}. \]
This is clearly a vector space over \C, in which the irreducible characters play a special role (see, e.g., \cite[Theorem 2.8 and Corollary 2.14]{Isaacs76}):
\begin{theorem} \label{orthonormal-thm} Under the inner product (\ref{inner-eq}), $\Irr(G)$ is an orthonormal basis of $\cf(G)$. In particular, if $f \in \cf(G)$ then $f$ has a unique orthonormal expansion \begin{align*}
   f &= \sum_{\chi \in \Irr(G)} \mu_{\chi} \chi,
\intertext{where the Fourier coefficients are given by}
   \mu_{\chi} &= \inner{f,\chi}.  \end{align*}
\end{theorem}

\begin{remark} \label{abelian-rem}
Two points should be noted when $G$ is abelian.  First, the conjugacy classes of $G$ are singletons, and so \emph{all} functions are class functions.  In this case $\Irr(G)$ is an orthonormal basis of $\C^G$, the set of all complex-valued functions on $G$.  We made use of this in Section~\ref{intro-sec}.

Second, $\Irr(G)$ is also a group under pointwise multiplication, and is isomorphic to $G$ itself.
In particular, we note the following:\begin{itemize}
  \item The irreducible characters may be indexed one-to-one by group elements.  This indexing is given explicitly in equation (\ref{cyclic_char-eq}) for the cyclic group $\Z_s$, and by (\ref{char_prod-eq}) when $G$ is a direct product of cyclic groups.  The same holds when $G$ is abelian.\footnote{Because of the Fundamental Theorem of Abelian Groups.}
  \item $\chi_u$ is principal iff $u$ is the identity of $G$.
\end{itemize}
We will use these facts in Lemma~\ref{weight-lemma}.
\end{remark}


\section{Strength}  \label{strength-sec}
If a design $D$ on the set $\cart{G}{1}{k}$ is displayed as columns of a $k \times N$ matrix, the \emph{projection} of $D$ on factors $i_1 < \cdots < i_m$ is the sub-matrix consisting of rows $i_1, \ldots, i_m$.  The resulting design $D'$ is a multisubset of $H = \cart{G}{i_1}{i_m}$, with counting function
  \begin{equation} \label{counting-eq} O'(\by) = \sum_{p(\bx) = \by} O(\bx),
  \end{equation}
where $p$ is the projection of $G$ on $H$ (namely, $p(x_1, \ldots, x_k) = (x_{i_1}, \ldots, x_{i_m}$)).
\begin{definition} \label{strength-def}  $D$ has \emph{strength $t\ge 1$}
if the projection of $D$ onto any $t$ factors has constant counting
function.
\end{definition}
In other words, for every $I=\{i_1,\ldots, i_t\}\subset \set{k}$, the projection $D'_I$ of $D$ on the
factors $i_1, \ldots, i_t$ consists of $\lambda_I$ copies of the full factorial $\cart{G}{i_1}{i_t}$, so that the counting function of $D'_I$ is the constant function $O'_I \equiv \lambda_I$.

We note that if $D$ has strength $t$ then it also has strength $t'$ for all $t' < t$.

When $G$ is a \emph{group}, the map $p$ projecting $G$ onto $H = \cart{G}{{i_1}}{{i_m}}$ is a group homomorphism.  Any such group $H$ has its own set of irreducible characters, of course.  Rephrasing
Definition~\ref{strength-def}, we see that \emph{$D$ has strength
$t$ iff whenever we project $G$ onto a factorial subgroup $H$ with at most $t$
factors, $O'$ is simply a multiple of the principal character of
$H$}. 

\section{Restatement of the theorem} \label{restatement-sec}
For a design $D$ based on the group $G = \cart{G}{1}{k}$ with counting function $O$, the $J$-characteristics (\ref{J-char-eq}) of $D$ are now given by
  \begin{equation} \label{J-char-nonabel-eq}
  \chi(D) =  \sum_{\bx \in G} O(\bx) \overline{\chi(\bx)}
  \end{equation}
for each $\chi \in \Irr(G)$.

As we have noted, when the group $G_i$ is abelian (and cyclic in particular), its irreducible characters may be indexed by $G_i$.  Without this, the concept of the weight of an element $\bu = (u_1, \ldots, u_k) \in G$ is no longer relevant, and so we must transfer this concept to the irreducible characters of $G$.

\begin{definition} \label{wt-def} For $\chi \in \Irr(G)$, let $K$ be the largest factorial subgroup contained in $\ker(\chi)$.  We define the \emph{base} of $\chi$ to be the factorial complement of $K$ in $G$, and the \emph{weight} of $\chi$ by
  \[ \wt(\chi) = \rank(\base(\chi)).\]
\end{definition}
Note that $\wt(\chi) = 0$ iff $\chi$ is the principal character of $G$.

If $K$ is as in Definition~\ref{wt-def}, then $\chi \equiv \chi(e)$ on $K$.  Now $e \in \base(\chi)$, so if $\chi \equiv 1$ on its base, then $\chi(e) = 1$, and so $\chi \equiv 1$. In other words, \emph{if $\chi$ restricted to its base is principal, then $\chi$ itself is principal}, and conversely.  (The converse is trivial.)

The following lemma relates Definition~\ref{wt-def} to the Hamming weight of elements $\bu \in G$ in the abelian case.  It makes use of the facts mentioned in Remark~\ref{abelian-rem}.
\begin{lemma} \label{weight-lemma}
Let $G = \cart{G}{1}{k}$ where $G_i$ is abelian for every $i$.  Fix an isomorphism indexing the irreducible characters by the elements of $G$.  Then $\wt(\chi_{\bu}) = \wt(\bu)$.
\end{lemma}
\begin{proof}  Given $\bu = (u_1, \ldots, u_k)$, let $I = \{i: u_i \neq e_i\}$, where $e_i$ is the identity of $G_i$.  For $i \notin I$, $\chi_{u_i} = \chi_{e_i} \equiv 1$, so
  \[ \chi_{\bu} = \prod_{i=1}^k \chi_{u_i} = \prod_{i \in I}\chi_{u_i}. \]
Let $K = \prod_{i \notin I} G_i$.  We claim that $K$ is the largest factorial subgroup of $G$ contained in $\ker(\chi_{\bu})$.  If so, then $\base(\chi_{\bu}) = \prod_{i \in I} G_i$, and so
  \[ \wt(\chi_{\bu}) = |I| = \wt(\bu).\]
(This still holds if $I = \emptyset$.)

To prove our claim, note that if $\bx \in K$ then $x_i = e_i$ for all $i \in I$, from which we have
  \[ \chi_{\bu}(\bx) = \prod_{i \in I} \chi_{u_i}(e_i) = 1 = \chi_{\bu}(\be), \]
so that $\bx \in \ker(\chi_{\bu})$.  Thus $K \subset \ker(\chi_{\bu})$.  To show that $K$ is the largest such factorial subgroup, consider $K' = K \times G_j$ for some $j \in I$.  Since $\chi_{u_j} \neq \chi_{e_j}$, we may choose $x_j \in G_j$ such that $\chi_{u_j}(x_j) \neq 1$.  Let
$\bx = (x_1, \ldots, x_k)$ where
  \[ x_i = \left\{   \begin{array}{ll}
                      x_j, & i = j \\
                      e_i, & i \neq j \\
                     \end{array}     \right. . \]
Then $\chi_{\bu}(\bx) = \chi_{u_j}(x_j) \neq 1 = \chi_{\be}(\bx)$, so $\bx \notin \ker(\chi_{\bu})$.  Thus $K'$ is not contained in $\ker(\chi_{\bu})$, which proves our claim.
\end{proof}

We now replace the definition of generalized wordlengths given in (\ref{genwordlength}) by
    \begin{equation} \label{genwordlength-nonabel}
    A_j(D) = N^{-2} \sum_{\wt(\chi) = j} |\chi(D)|^2 \qquad \mbox{for} \; j =
  1, \ldots, k, \end{equation}
where $N = |D|$, defined as in (\ref{|D|-eq}).  With this, we restate our theorem as follows:
\begin{theorem} \label{strength-thm}
Let $D$ be a fractional factorial design on $G = \cart{G}{1}{k}$ with counting
function $O$, and assume $O$ is a class function on $G$. For each $\chi
\in \Irr(G)$ define $\chi(D)$ by (\ref{J-char-nonabel-eq}), 
and let
$\mu_{\chi} = \inner{O,\chi}$. Define $A_j(D)$ by
(\ref{genwordlength-nonabel}), and assume $t \ge 1$. Then the
following are equivalent:
\begin{enumerate}
\item  $D$ has strength $t$.  \label{strength=t}
\item  $A_1(D) = \cdots = A_t(D) = 0$.  \label{Ai=0}
\item  $\mu_{\chi} = 0$ for all $\chi \in \Irr(G)$ with $1 \le \wt(\chi) \le t$. \label{mu=0}
\end{enumerate}

\end{theorem}

In Section~\ref{examples-sec} we give two nonabelian examples with counting functions that are class functions. 


\section{Proof of the theorem} \label{proof-sec}
As in the abelian case, we have
  \[ \mu_{\chi} = \inner{O, \chi} = \frac{1}{|G|}\sum_{\bx \in G} O(\bx)\overline{\chi(\bx)} = \frac{1}{|G|}\chi(D) \]
for each $\chi \in \Irr(G)$, so that the generalized wordlengths (\ref{genwordlength-nonabel}) are given
by
  \[ A_j(D) = N^{-2} \sum_{\wt(\chi) = j} |\chi(D)|^2
  =  \frac{|G|^2}{N^2}\sum_{\wt(\chi)=j} |\mu_{\chi}|^2. \]
Thus we immediately have the equivalence of (\ref{Ai=0}) and (\ref{mu=0}) in Theorem~\ref{strength-thm}. Our goal is to prove the equivalence of (\ref{strength=t}) and (\ref{mu=0}).

We noted in Section~\ref{strength-sec} that $D$ has strength $t$ iff whenever we project $G$ onto a factorial subgroup $H$ of rank at most $t$,  the counting function $O'$ of the projected design is a simply a multiple of the principal character of $H$.  Assuming that $O'$ is a class function on $H$, we have the orthonormal expansion
  \begin{equation} \label{O'-expansion-eq} O' = \sum_{\chihat \in \Irr(H)} \mu_{\chihat} \chihat \end{equation}
from which we see that $D$ has strength $t$ iff, for the projection on any $H$ with $\rank(H) \le t$, the Fourier coefficients $\mu_{\chihat} = \inner{O',\chihat}$ vanish for all non-principal irreducible characters $\chihat$ of $H$.

On the other hand, when $O$ is a class function on $G$ we have
  \begin{equation} \label{O-expansion-eq} O = \sum_{\chi \in \Irr(G)} \mu_{\chi} \chi. \end{equation}
Thus the proof requires a comparison of equations (\ref{O'-expansion-eq}) and (\ref{O-expansion-eq}).  It rests on the following two lemmas.  In both, $G$ is an arbitrary finite group, and
we denote the coset $Kg$ by \gbar.\footnote{We also use the ``bar" notation to indicate complex conjugates; context will determine which is meant.}

\begin{lemma} \label{quotient-lemma} \cite[Lemma 2.22]{Isaacs76}
Let $K$ be a normal subgroup of $G$.
\begin{enumerate}
\item \label{going-down} If $\chi$ is a character of $G$ and $K\subseteq \ker(\chi)$,
then $\chi$ is constant on cosets of $K$ in $G$ and the function
$\chihat$ on $G/K$ defined by $\chihat(\gbar) = \chi(g)$ is a character
of $G/K$.

\item \label{going-up} If $\chihat$ is a character of $G/K$, then the function
$\chi$ defined by $\chi(g) = \chihat(\gbar)$ is a character of $G$ and $K \subseteq \ker(\chi)$.

\item In both (\ref{going-down}) and (\ref{going-up}), $\chi\in \Irr(G)$
iff $\chihat\in \Irr(G/K)$.
\end{enumerate}
\end{lemma}


\begin{lemma} \label{f'-lemma} Let $K$ be normal in $G$ and let $H = G/K$.  Let $f \in \C^G$ and $\chi \in \Irr(G)$. Define $f' \in \C^H$ by
  \begin{equation} \label{f'-def-eq} f'(\ybar) = \sum_{x \in \ybar} f(x), \end{equation}
and define $\chihat$ as in Lemma~\ref{quotient-lemma}(\ref{going-down}).
If $f \in \cf(G)$ then $f' \in \cf(H)$, and
 \[ \inner{f',\chihat} = |K|\inner{f,\chi}. \]
\end{lemma}

Note that when $G$ is a group, the counting function $O'$ of a projected design, defined in (\ref{counting-eq}), is of the form (\ref{f'-def-eq}).

\begin{proof}[Proof of Lemma \ref{f'-lemma}]  First, suppose that $\ybar_1$ and $\ybar_2$ are conjugate in $H$.  Then
  \[ \ybar_2 =  (Kh\inv)(Ky_1)(Kh) = h\inv Ky_1 h \]
for some $h \in G$, so the elements of the cosets $\ybar_1$ and $\ybar_2$ may be paired in such a way that each $x_2 \in \ybar_2$ is the conjugate of a unique $x_1 \in \ybar_1$.  Since $f$ is a class function, $f(x_1) = f(x_2)$, so $\sum_{x \in \ybar_1} f(x) = \sum_{x \in \ybar_2} f(x)$.  This shows that $f'$ is a class function on $H$.

By Lemma~\ref{quotient-lemma}(\ref{going-down}), $\chi$ is constant on each
coset \ybar, and $\chihat(\ybar) = \chi(y)$. We then have
\begin{eqnarray*}
\inner{f, \chi} & = & \frac{1}{|G|}\sum_{x\in G} f(x) \overline{\chi(x)} \\
   & = & \frac{1}{|G|}\sum_{\ybar \in H} \sum_{x\in \ybar} f(x) \overline{\chi(x)} \\
   & = & \frac{1}{|G|}\sum_{\ybar \in H} \overline{\chihat(\ybar)} \sum_{x\in \ybar} f(x) \\
   & = & \frac{1}{|G|}\sum_{\ybar \in H} f'(y) \overline{\chihat(\ybar)}
    =  \frac{\inner{f', \chihat}}{|K|}.
\end{eqnarray*}

\end{proof}

We are now ready to complete the proof of Theorem~\ref{strength-thm}.  We begin by noting two things.  First, according to Lemma~\ref{f'-lemma}, the assumption that $O$ is a class function guarantees that the counting function $O'$ of \emph{every} projected design is also a class function.  In particular, $O'$ has an orthonormal expansion (\ref{O'-expansion-eq}).

Second, when $G = H \times K$ is a direct product, $H$ is isomorphic to $G/K$, and the character \chihat\ in Lemma~\ref{quotient-lemma} is the restriction of $\chi$ to $H$.  Recall that $\chi$ is nonprincipal iff its restriction to its base is nonprincipal.

(\ref{strength=t}) $\Rightarrow$ (\ref{mu=0}):  Assume that $D$ has strength $t$, and let $\chi \in \Irr(G)$ with $1 \le \wt(\chi) \le t$.  We need to show that the Fourier coefficient $\mu_{\chi}$ of $O$ vanishes.  Let $H = \base(\chi)$, and let \chihat\ be defined by $\chi$ as in Lemma~\ref{quotient-lemma}(\ref{going-down}) where $K$ is the complement of $H$ in $G$. Now $\chi$ is nonprincipal, as $\wt(\chi) \ge 1$, so \chihat\ is as well.  On the other hand, since $H$ has at most $t$ factors and $D$ has strength $t$, $O'$ is a multiple of the principal character of $H$. But then $\mu_{\chihat} = \inner{O',\chihat} = 0$, and so by Lemma~\ref{f'-lemma} $\mu_{\chi} = \inner{O,\chi} = 0$.

(\ref{mu=0}) $\Rightarrow$ (\ref{strength=t}):  Assuming the
condition on the coefficients $\mu_{\chi}$ given by (\ref{mu=0}), we
must show that $D$ has strength $t$.  To this end, consider any
factorial subgroup $H$ of $G$ having at most $t$ factors, let K be its factorial
complement, and let $O'$ be the counting function of the design projected on $H$.  Let $\chihat
\in \Irr(H)$ be a nonprincipal character on $H = G/K$, and let $\chi \in \Irr(G)$ correspond to it via
Lemma~\ref{quotient-lemma}(\ref{going-up}).  In particular, $\chi$ is nonprincipal and $K \subseteq \ker(\chi)$. Let $K_1$ be the largest factorial subgroup contained in $\ker(\chi)$, so that $K_1 \supseteq K$.  Taking complements, we have $\base(\chi) \subseteq H$, so that
  \[ \wt(\chi) = \rank(\base(\chi)) \le \rank(H) \le t. \]
Therefore, $\inner{O,\chi} = \mu_{\chi} = 0$ by assumption.  But then
$\inner{O',\chihat} = 0$ as well, by Lemma~\ref{f'-lemma}, so $O'$ must be a multiple of the principal character of $H$.  Since this holds for all such $H$, $D$ has strength $t$.

\section{Two examples} \label{examples-sec}
We conclude by giving two examples of designs whose treatment combinations are indexed by nonabelian groups and whose counting functions are class functions of those groups.  Both examples make use of $S_3$, the symmetric group on 3 letters.  We write
  \[ S_3 = \{ e, a, b, c, x, y \} \]
where $e$ is the identity, $a, b$ and $c$ are transpositions, and $x$ and $y$ are 3-cycles.  As is well known, the conjugacy classes of $S_3$ are $\{e\}, \{a,b,c\}$ and $\{x,y\}$.

We also make use of the facts that the conjugacy classes of an abelian group are the singleton subsets, and that in a direct product, $(x_1, \ldots, x_k)$ and $(y_1, \ldots, y_k)$ are conjugate iff $x_i$ and $y_i$ are conjugate for each $i$.

\begin{example} \textbf{A 1/2-fraction of $6 \times 2 \times 2$ experiment of strength 2.}

We index the treatment combinations by $G = S_3 \times \Z_2 \times \Z_2$.  The following array displays the runs as columns, the vertical lines separating conjugacy classes.

\[ D = \left[ \begin{array}{c|cc|c|cc|ccc|ccc}
e & x & y & e & x & y & a & b & c & a & b & c \\
1 & 1 & 1 & 0 & 0 & 0 & 0 & 0 & 0 & 1 & 1 & 1 \\
0 & 0 & 0 & 1 & 1 & 1 & 0 & 0 & 0 & 1 & 1 & 1
\end{array} \right] \]

This makes use of 6 of the 12 conjugacy classes of $G$.  The other 6 classes would furnish another example.  Since $S_3$ is the smallest nonabelian group, this is the smallest non-trivial fractional factorial design of strength 2 that can be indexed by a nonabelian group.
\end{example}

\begin{example}  \textbf{A 1/2-fraction of a $6 \times 2 \times 2 \times 2$ experiment of strength 3.}
\[ D = \left[ \begin{array}{c|c|c|c|cc|cc|cc|cc|ccc|ccc|ccc|ccc}
e & e & e & e & x & y & x & y & x & y & x & y & a & b & c & a & b & c & a & b & c & a & b & c  \\
0 & 0 & 1 & 1 & 0 & 0 & 0 & 0 & 1 & 1 & 1 & 1 & 0 & 0 & 0 & 0 & 0 & 0 & 1 & 1 & 1 & 1 & 1 & 1   \\
0 & 1 & 0 & 1 & 0 & 0 & 1 & 1 & 0 & 0 & 1 & 1 & 0 & 0 & 0 & 1 & 1 & 1 & 0 & 0 & 0 & 1 & 1 & 1   \\
0 & 1 & 1 & 0 & 0 & 0 & 1 & 1 & 1 & 1 & 0 & 0 & 1 & 1 & 1 & 0 & 0 & 0 & 0 & 0 & 0 & 1 & 1 & 1
\end{array} \right] \]
We have indexed the treatment combinations by $G = S_3 \otimes \Z_2 \otimes \Z_2 \otimes \Z_2$.
Note that the last three rows consist of three copies of the full $2^3$ factorial design, split into its two regular $2^{3-1}$ fractions given by the solutions $(X,Y,Z)$ of $X+Y+Z = 0$ and = 1 modulo 2. We have attached the first fraction to $e, x,$ and $y$, and the second fraction to $a, b$ and $c$.
\end{example}

Further examples are given in \cite{Maggiethesis12} 

\bibliography{nonsimple}
\bibliographystyle{plain}

\end{document}